\documentclass[12pt]{article}

\usepackage{amsthm}
\usepackage{fullpage}
\usepackage{amssymb, amsmath}
\usepackage{hyperref}
\usepackage{graphicx}

\newtheorem{theorem}{Theorem}

\newtheorem{lemma}{Lemma}

\newtheorem{corollary}{Corollary}

\newtheorem{proposition}{Proposition}

\theoremstyle{definition}
\newtheorem*{definition}{Definition}

\theoremstyle{remark}

\def\boxit#1{\medskip\vbox{\hrule \hbox{\vrule\kern15pt\vbox{\kern5pt
\vbox{\advance\hsize -30pt #1\par} \kern5pt}\kern15pt\vrule}\hrule} \medskip}
\renewcommand{\P}{\mathcal{P}}
\newcommand{\M}{\mathcal{M}}
\renewcommand{\L}{\mathcal{L}}
\newcommand{\floor}[1]{\left\lfloor#1\right\rfloor}

\newcommand{\Mod}[1]{\ (\mathrm{mod}\ #1)}

\title{Partitioning the power set of $[n]$ into $C_k$-free parts}

\author{Eben Blaisdell,$^{1}$ Andr\'as Gy\'arf\'as,$^{2}$ Robert A. Krueger,$^{3}$ Ronen Wdowinski$^{4}$}

\begin{document}

\maketitle
% \noindent\footnotetext[1]{This article was written under the auspices of the Budapest Semesters in Mathematics program.}
\noindent\footnotetext[1]{Department of Mathematics, Bucknell University, Lewisburg, Pennsylvania. \texttt{emb038@bucknell.edu}.}%
\noindent\footnotetext[2]{Alfr\'ed R\'enyi Institute of Mathematics, Hungarian Academy of Sciences, Budapest, P.O. Box 127, Budapest, Hungary, H-1364. \texttt{gyarfas.andras@renyi.mta.hu}.}%
\noindent\footnotetext[3]{Department of Mathematics, Miami University, Oxford, Ohio. \texttt{kruegera@miamioh.edu}.}%
\noindent\footnotetext[4]{Department of Mathematics, Rice University, Houston, Texas. \texttt{rmw5@rice.edu}.}%
\begin{abstract}
We show that for $n \geq 3, n\ne 5$, in any partition of $\mathcal{P}(n)$, the set of all subsets of $[n]=\{1,2,\dots,n\}$, into $2^{n-2}-1$ parts, some part must contain a triangle --- three different subsets $A,B,C\subseteq [n]$ such that $A\cap B$, $A\cap C$, and $B\cap C$ have distinct representatives. This is sharp, since by placing two complementary pairs of sets into each partition class, we have a partition into $2^{n-2}$ triangle-free parts.  We also address a more general Ramsey-type problem: for a given graph $G$, find (estimate) $f(n,G)$, the smallest number of colors needed for a coloring of $\mathcal{P}(n)$, such that no color class contains a Berge-$G$ subhypergraph. We give an upper bound for $f(n,G)$ for any connected graph $G$ which is asymptotically sharp (for fixed $k$) when $G=C_k, P_k, S_k$, a cycle, path, or star with $k$ edges. Additional bounds are given for $G=C_4$ and $G=S_3$.
\end{abstract}

\section{Introduction and results}

Hypergraph Ramsey problems usually address the existence of large monochromatic structures in colorings of the edges of $K_n^r$, the complete $r$-uniform hypergraph. It is rare that monochromatic structures are sought in colorings of hypergraphs containing {\em all subsets} of $[n]$, $\P(n)$. An exception is the Finite Unions Theorem of Folkman, Rado, Sanders \cite{GRS}. A more recent research in this direction is by Axenovich and Gy\'arf\'as \cite{AGY}, where Ramsey numbers of Berge-$G$ hypergraphs were studied for several graphs $G$ in colorings of $\P(n)$.  Ramsey numbers of Berge-$G$ hypergraphs in the {\em uniform case} have been investigated also in \cite{GMOV,STWZ}.

A hypergraph $H=(V,F)$ is called Berge-$G$ if $G=(V,E)$ is a graph and there exists a bijection $g: E(G)\mapsto E(H)$ such that for $e\in E(G)$ we have $e\subseteq g(e)$.
Note that for a given graph $G$ there are many Berge-$G$-hypergraphs. Berge-$G$ hypergraphs were defined by Gerbner and Palmer \cite{GP} to extend the notion of paths and cycles in hypergraphs introduced by Berge in \cite{B}. In particular, a Berge-$C_3$ hypergraph consists of three subsets $A,B,C\subseteq [n]$ such that $A\cap B,A\cap C,B\cap C$ have distinct representatives. When there is no confusion, we will often refer to a Berge-$G$ hypergraph simply as `a $G$.' The graphs $C_k,P_k,S_k$ denote cycle, path, and star with $k$ edges, respectively. It is customary to use the names \emph{triangle} and \emph{claw} for the graphs $C_3$ and $S_3$, respectively.

A hypergraph $H$ with vertex set $[n]$ and whose edges are sets from $\P(n)$ is called {\em $G$-free}, if it does not contain any subhypergraph isomorphic to a Berge-$G$ hypergraph. The {\em intersection graph} of a hypergraph $H$ is a graph $G$ whose vertices represent edges of $H$ and where there is an edge in $G$ if and only if the corresponding edges of $H$ have non-empty intersection. Note that if the intersection graph of $H$ has no subgraph isomorphic to the intersection graph of $G$ (that is, the line graph of $G$), then $H$ is $G$-free. The reverse statement is not true: the intersection graph of the hypergraph $H$ with edges $\{1,2\},\{1,2,3\},\{1,2,4\}$ is a triangle but $H$ is triangle-free.

To define the Ramsey-type problem we address here, let $f(n,G)$ be the smallest number of colors in a coloring of $\P(n)$ such that all color classes are $G$-free. In other words, in every coloring of $\P(n)$ with $f(n,G)-1$ colors, there is a Berge-$G$ subhypergraph in some color class. We use the terms {\em coloring, partitioning} of $\P(n)$ in the same sense. Since the presence of empty sets and singleton sets do not influence whether a coloring is $G$-free, we usually construct colorings of $\P^*(n)$, what we define to be $\P(n)$ with the empty set and the singletons removed. However, the following natural partition of the whole power set of $[n]$ is useful.
For every $A\subseteq [n-1]$, the part defined by $A$ is
$$\{X_1(A)=A,X_2(A)=[n]\setminus X_1(A), X_3(A)=A\cup \{n\}, X_4(A)=[n]\setminus X_3(A)\}.$$
Since $A$ and $[n-1]\setminus A$ define the same part, we have $2^{n-2}$ parts (each of size four).
This partition was used in \cite{AGY} to show that $f(n,C_3)\le 2^{n-2}$. Observing that
$$X_1(A)\cap X_2(A)=X_3(A)\cap X_4(A)=X_1(A)\cap X_4(A)=\emptyset,$$
these parts are $C_3$-free, $C_4$-free and $S_3$-free.
Thus we have a natural upper bound for three small graphs:
\begin{proposition}\label{basepart} $f(n,G)\le 2^{n-2}$ for $G\in \{C_3,C_4,S_3\}$.
\end{proposition}
How sharp is this upper bound for the three small graphs involved? The easiest lower bound comes for the claw.
\begin{proposition}\label{claw} $2^{n-2} - n/2 \leq f(n, S_3)$. In general, $\frac{2^{n-1}}{k-1} - O\left(n^{k-2}\right) \le f(n,S_k)$.
\end{proposition}
\begin{proof} Consider a partition $Q$ of $\P(n)$ into $S_k$-free parts.  Let $H=(V,E)$  be the subhypergraph of $\P(n)$ determined by the edges of size at least $k$. Then $Q$ partitions $H$ into $S_k$-free parts $H_i=(V,E_i)$, for $i=1,\dots,t$.  Since $k$ edges of size at least $k$ cannot have common intersection by the $S_k$-free property, each hypergraph $H_i$ has maximum degree at most $k-1$.  Therefore

$$n2^{n-1} - \left(n + 2\binom{n}{2} + \cdots + (k-1)\binom{n}{k-1}\right) =\sum_{v\in V} d_H(v) = \sum_{i=1}^t \sum_{v\in V} d_{H_i}(v) \le (k-1)nt,$$ implying $t \geq \frac{2^{n-1}}{k-1} - \frac{1}{k-1}\left( 1 + \binom{n-1}{1} + \cdots + \binom{n-1}{k-2} \right) = \frac{2^{n-1}}{k-1} - O\left(n^{k-2}\right)$. For $k=3$, this calculation gives $2^{n-2} - n/2 \leq f(n, S_3)$.
\end{proof}

The discrepancy of $-n/2$ between Proposition \ref{basepart} and \ref{claw} for $f(n,S_3)$ is the consequence of the fact that three edges of $\P^*(n)$ intersecting in a vertex $v$ do not define a claw in the special case when the three edges are $\{v,x,y\},\{v,x\},\{v,y\}$.
Utilizing this with several different designs, we have small examples in Section \ref{designs} showing that the upper bound for $f(n, S_3)$ in Proposition \ref{basepart} can sometimes be lowered (in particular, we show that $f(6,S_3)\le 15$ and $f(9,S_3)\le 126$). It is unclear whether one can use this phenomenon to decrease the upper bound for infinitely many $n$.

For the case of the triangle, the upper bound of Proposition \ref{basepart} is tight. For odd $n\ge 7$ this was shown  with a simple proof in \cite{AGY}. Somewhat surprisingly, this remains true for the even $n$ case as well (but not for $n=5$).
\begin{theorem}\label{trianglethm}
For $n\ge 3, n\ne 5$, $f(n,C_3)=2^{n-2}$. Additionally, $f(5, C_3) = 7$.
\end{theorem}
In case of $G=C_4$ we improve the upper bound of Proposition \ref{basepart} by a constant factor and slightly improve the lower bound ${2^{n-1}\over 3}(1-o(1))$ from \cite{AGY}.

\begin{theorem}\label{c4theorem} For even $n$, we have $f(n, C_4) = \frac{2^{n-1}}{3}\left(1 + \Theta\left(\frac{1}{\sqrt{n}}\right)\right)$. Additionally, for all $n \geq 27$, we have $\frac{2^{n-1}}{3} \leq f(n,C_4) \leq {2^{n-1}\over 3} \left(1 + O\left(\frac{1}{\sqrt{n}}\right)\right)$.
\end{theorem}

While our lower bound for $f(n, C_4)$ for even $n$ is asymptotically larger than our lower bound for odd $n$, we have no reason to believe that the lower bound for odd $n$ cannot be improved. We suspect that a better bound for odd $n$ would follow from a similar proof as that with even $n$, just with more work involved.

For the upper bound on $f(n,C_4)$ we combine designs to include almost all sets in $\P(n)$.
In fact, we do this to provide an upper bound for $f(n,G)$ \emph{for any connected graph $G$}. The construction is based on {\em asymptotically optimal packings}, $D(n,m,r)$, which is a large subset $S\subseteq {[n]\choose m}$ with the property that every $r$-element subset of $[n]$ is contained in {\em at most one} member of $S$. The existence of such packings was proved in a breakthrough paper of R\"odl \cite{RO}. For our purposes only a special case is needed, $D(n,m,m-1)$, where constructions were known earlier, for example in \cite{KUZ}.
\begin{theorem}\label{asymcycle}
Let $G$ be a connected graph with $k$ edges, where $k \geq 2$ is fixed. Then $f(n, G) \leq \frac{2^n}{2(k-1)} \left(1 + O\left(\frac{1}{\sqrt{n}}\right)\right)$.
\end{theorem}

The upper bound of Theorem \ref{asymcycle} gives the upper bound in Theorem \ref{c4theorem}. In fact, it also matches the corresponding asymptotic lower bound ${2^{n-1}\over |E(G)|-1}(1-o(1))$ in \cite{AGY} when $G$ is a cycle or path, and the asymptotic lower bound of Proposition \ref{claw}, implying
\begin{corollary} $\frac{2^n}{2(k-1)} (1 - o(1))\leq f(n, C_k), f(n, P_k), f(n, S_k) \leq \frac{2^n}{2(k-1)} (1 + o(1))$.
\end{corollary}

\section{Proof of Theorem \ref{trianglethm}}\label{tr}

It was shown in \cite{AGY} that $f(n,C_3)=2^{n-2}$ for any odd $n\ge 3, n\ne 5$.  The following $C_3$-free partition of $\P^*(5)$ shows that $n=5$ is indeed exceptional. (Here and later we represent sets of small numbers without commas and brackets.)
$$X_1=\{[5],[4]\}, Y_1=\{124,234,245\}, Y_2=\{123,134,135\}, Z_1=\{12,35,1235,345\},$$
\begin{equation}\label{f5}
 Z_2=\{23,45,2345,145\}, Z_3=\{34,15,1345,125\}, Z_4=\{14,25,1245,235\}.
\end{equation}

In fact, (\ref{f5}) is the only partition of $\P^*(5)$ into at most seven $C_3$-free parts (up to permutations), implying
 $f(5,C_3)=7$. For $n \neq 5$, three sets of size at least $\lfloor n/2 \rfloor +1$ always form a triangle (this is proven for odd $n$ in \cite{AGY} and generalized for even $n$ in Lemma \ref{trianglelemma}). This is indeed not true for $n=5$, as witnessed by the `crowns' $Y_1$ and $Y_2$ in (\ref{f5}).

Let $\L$ be the set of all subsets of $[n]$ of size at least $\lfloor n/2\rfloor + 1$ (these are the `large' subsets). For even $n$ let $\M$ be the set of all subsets of $[n]$ of size $n/2$ (these are the `medium' subsets). Note that $2|\L| + |\M| = 2^n$.
%(In the following proof, `WLOG' stands for `without loss of generality.')

\begin{lemma}\label{trianglelemma}
For every even $n \geq 6$, we have the following:

\begin{enumerate}

\item For any distinct $M_1,M_2,M_3,M_4,M_5 \in \M$, some three form a triangle.

\item For any distinct $M_1,M_2,M_3 \in \M, L \in \L$, some three form a triangle.

\item Any distinct $M \in \M, L_1,L_2 \in \L$ form a triangle.

\item Any distinct $L_1, L_2, L_3 \in \L$ form a triangle.

\end{enumerate}
\end{lemma}

\begin{proof}[Proof of Theorem \ref{trianglethm} from Lemma \ref{trianglelemma}]
For odd $n$ the theorem  was proved in \cite{AGY}.  By Proposition \ref{basepart}, we have to prove that $f(n,C_3)\ge 2^{n-2}$ for even $n$. Let $n\geq 6$, and let $Q$ be a partition of $\P(n)$ into the minimum number of triangle-free parts. (For $n=4$ a similar lemma and argument works.) Let there be $a$ parts of $Q$ with exactly two sets of $\L$, let there be $b$ parts of $Q$ with exactly one set of $\L$, and let there be $c$ parts of $Q$ with no sets of $\L$. Lemma \ref{trianglelemma} implies that these account for all the parts, so that $a + b + c = f(n, C_3)$. Lemma \ref{trianglelemma} also implies that $|\M| \leq 2b + 4c$, and since $|\L| = 2a + b$, we have
\[ f(n, C_3) = a + b + c \geq \frac{1}{4}(2|\L| + |\M|) = \frac{1}{4}(2^n)=2^{n-2} .\]
\end{proof}

\begin{proof}[Proof of Lemma \ref{trianglelemma}]

Since a set of $\L$ always contains as a subset a set of $\M$, it is clear that statement 3 implies statement 4. Thus we only need to prove statements 1, 2, and 3.

Let's first note some basic intersection properties of sets from $\M \cup \L$. Let $L_1, L_2 \in \L$ and $M_1, M_2, M_3 \in \M$ be arbitrary. It is clear that $|L_1 \cap L_2| \geq 2$, $|L_1 \cap M_1| \geq 1$, and either $|M_1 \cap M_2| \geq 1$ or $|M_2 \cap M_3| \geq 1$.

In any of the three cases of the lemma, we first want to find three pairwise intersecting sets. In the first case, WLOG $M_1$ intersects with $M_2$, $M_3$, and $M_4$, and again WLOG $M_2$ intersects with $M_3$. In the second case, $L$ intersects $M_1$, $M_2$, and $M_3$, and WLOG $M_1$ intersects $M_2$. In the third case, every pair of sets intersect.

Let $A, B, C \in \M \cup \L$ be three distinct pairwise intersecting sets, in any case, and suppose they do not form a triangle. By Hall's theorem as applied to distinct representatives, there are only a few cases where they may not form a triangle. WLOG, either $|(A \cap B) \cup  (A \cap C)| \leq 1$ or $|(A \cap B) \cup (A \cap C) \cup (B \cap C)| \leq 2$. In the first case, it cannot be that $|A \cap B \cap C| = 0$, since the sets are pairwise intersecting, so we must have $|A \cap B \cap C| = 1$ and $|A \cap B \setminus C| = |A \cap C \setminus B| = 0$. In the second case, it likewise cannot be that $|A \cap B \cap C| = 0$. The case where $|A \cap B \cap C| = 1$ falls into the previous case, so this case reduces to $|A \cap B \cap C| = 2$ and $|A \cap B \setminus C| = |B \cap C \setminus A| = |C \cap A \setminus B| = 0$.

Define $\delta_A = |A|-n/2$, and likewise for $B$ and $C$. Furthermore let $\delta = \delta_A + \delta_B + \delta_C$.

\textbf{Case 1:} $|A \cap B \cap C| = 2$ and $|A \cap B \setminus C| = |B \cap C \setminus A| = |C \cap A \setminus B| = 0$. Here we count
\[ n \geq |A \cup B \cup C| = |A| + |B| + |C| - |A \cap B| - |A \cap C| - |B \cap C| + |A \cap B \cap C| \geq \frac{3}{2}n + \delta - 2 - 2 - 2 + 2 \]
implying
\[ n + 2\delta \leq 8 .\]

\textbf{Case 1a:} Suppose $n=8$. Then $\delta = 0$ and so $A, B, C \in \M$ and WLOG the configuration is isomorphic to $A = 1234$, $B =1256$, and $C = 1278$. A fourth set $D \in \M\cup \L$ must meet two of $A,B,C$ in a vertex not in $\{1,2\}$, forming a triangle with them.

\textbf{Case 1b:} Suppose $n=6$. If $\delta=0$, then WLOG $A = 123$, $B = 124$, and $C = 125$. The only pairs of vertices a fourth set $D \in \M \cup \L$ may contain without forming a triangle are those pairs containing $6$ and the pair $12$. Thus unless $D = 126$, we have a triangle. If $D = 126$, then we must be in the first case of the lemma, and so we may take a fifth set $E \in \M$. Since $\{1,2\} \not\subseteq E$, we have that $E$ must contain two vertices of $\{3,4,5,6\}$, forming a triangle.

If $\delta = 1$, then WLOG we are in the second case of the lemma and $A = 123$, $B = 124$, and $C = 1256$. Since the fourth set $D \in \M$ contains some pair of vertices other than $12$ and $56$, we have a triangle.

\textbf{Case 2:} $|A \cap B \cap C| = 1$ and $|A \cap B \setminus C| = |A \cap C \setminus B| = 0$. Here we count
\[ n \geq |A \cup B \cup C| = |A| + |B| + |C| - |A \cap B| - |A \cap C| - |B \cap C| + |A \cap B \cap C| \geq \frac{3}{2}n + \delta - |B \cap C| - 1 ,\]
therefore
\[ \frac{1}{2} n + \delta - 1 \leq |B \cap C| \leq n - |A| \quad \text{(since $B$ and $C$ are distinct)} ,\]
implying
\[  2\delta_A + \delta_B + \delta_C \leq 1 .\]
Thus $\delta_A = 0$, and at most one of $\delta_B$ and $\delta_C$ is $1$, meaning that WLOG $A, B \in \M$, so we are not in the third case of the lemma. This means that the third case of the lemma was proved in case 1, so we are free to use it to finish the proof here. Let $D \in \M$ be a set distinct from $A$, $B$, and $C$.

If $D \subseteq B \cup C$, then $B$, $C$, and $D$ are three sets of size at least $(n-2)/2 + 1$ contained within a set of size $n/2+1 \leq n-2$. We may apply the third case of the lemma (with $n-2$ for $n$) to see that $B, C, D$ form a triangle.

Otherwise, let $x \in A \cap B \cap C$, let $y \in D \cap A \setminus \{x\}$, and let $z \in D \cap (B \cup C) \setminus \{x\}$. Note that $x,y,z$ necessarily exist and are distinct, so either $A, D, B$ or $A, D, C$ form a triangle.

\end{proof}

\section{Proof of Theorem \ref{c4theorem}}\label{c4}

The upper bound of Theorem \ref{c4theorem} follows from Theorem \ref{asymcycle} (with $k=4$). So we prove the lower bound. As in Section \ref{tr}, we need a lemma concerning sets of $\M \cup \L$. We also give the corresponding lemma for odd $n$, which we prove from Lemma \ref{evenc4lemma}.

\begin{lemma}\label{evenc4lemma} For every even $n \ge 26$, we have the following:
	\begin{enumerate}
		\item For any distinct $M_1, M_2, M_3, M_4, M_5 \in \M$, some four form a $C_4$.
		\item For any distinct $M_1, M_2, M_3, M_4 \in \M$, $L_1 \in \L$, some four form a $C_4$.
		\item Any distinct $M_1, M_2 \in \M$, $L_1, L_2 \in \L$ form a $C_4$.
		\item Any distinct $M_1 \in \M$, $L_1, L_2, L_3 \in \L$ form a $C_4$.
	\end{enumerate}
\end{lemma}

\begin{lemma}\label{oddc4lemma} For every odd $n \geq 27$, any distinct $L_1, L_2, L_3, L_4 \in \L$ form a $C_4$.
\end{lemma}

\begin{proof}[Proof of Theorem \ref{c4theorem} from Lemmas \ref{evenc4lemma} and \ref{oddc4lemma}]

Let $n \ge 26$ be even, and let $Q$ be a partition of $\P(n)$ into the minimum number of $C_4$-free parts. Say $Q$ has $a$ parts with three sets in $\L$, $b$ parts with two sets in $\L$, $c$ with one, and $d$ with no sets in $\L$.  Lemma \ref{evenc4lemma} implies that these account for all the parts of $Q$, so $a + b + c + d = f(n, C_4)$. Moreover, Lemma \ref{evenc4lemma} implies the relations $|\L| = 3a + 2b + c$ and $|\M| \le b + 3c + 4d$. Since $|\M| = \binom{n}{n/2} = \Theta(2^n/\sqrt{n})$, this gives us (by $b + 3c + 4d \le {3b\over 2}+3c+{9d\over 2})$ that
\[ f(n, C_4) = a + b + c + d \ge \frac{1}{6}\left(2|\L| + \frac{4}{3}|\M|\right) = \frac{1}{6}\left(2^n + \frac{1}{3} |\M|\right) = \frac{2^{n-1}}{3}\left(1 + \Theta\left(\frac{1}{\sqrt{n}}\right)\right) .\]

For odd $n \geq 27$, again take such a minimal $C_4$-free partition of $\P(n)$. Each part has at most three sets in $\L$, so $f(n, C_4) \geq \frac{1}{3}|\L| = \frac{2^{n-1}}{3}$.

\end{proof}

In order to prove Lemma \ref{evenc4lemma}, we need the following definition:

\begin{definition}

Assume $n \ge 4$ is even. We say that four distinct sets $A, B, C, D \in \M \cup \L$ form a \emph{$\Psi$-configuration} if there exists some $x$ such that $A \cap B, A \cap C, A \cap D \subseteq \{x\}$. In such a configuration we call $A$ a \emph{stem}.

\end{definition}

Let us elaborate on the structure of a $\Psi$-configuration $A,B,C,D$. Suppose $A$ is a stem and $A \cap (B \cup C \cup D) \subseteq \{x\}$. Since $A,B,C,D$ are distinct sets in $\M \cup \L$, we have the inequalities $|A| \ge \frac{n}{2}$ and $|B \cup C \cup D| \ge \frac{n}{2} + 1$. But also,
\[n+1 \ge |A \cup (B \cup C \cup D)| + |A \cap (B \cup C \cup D)| = |A| + |B \cup C \cup D| \ge \frac{n}{2} + (\frac{n}{2} + 1) = n+1.\]
So in fact, $|A| = \frac{n}{2}$ and $|B \cup C \cup D| = \frac{n}{2} + 1$. That is to say, $A \in \M$, and $B,C,D$ are $\frac{n}{2}$- or $(\frac{n}{2}+1)$-subsets of the $(\frac{n}{2} + 1)$-set $([n] \setminus A) \cup \{x\}$. Based on this, it is easy to see that a stem of a $\Psi$-configuration is unique.

Also note that in this $\Psi$-configuration we have
\[|B \cap C| = |B| + |C| - |B \cup C| \ge \frac{n}{2} + \frac{n}{2} - (\frac{n}{2} + 1) = \frac{n}{2} - 1,\]
and similarly $|B \cap D|, |C \cap D| \ge \frac{n}{2} - 1$. Thus, the non-stem sets of a $\Psi$-configuration pairwise intersect in at least $\frac{n}{2} - 1$ elements. Finally, observe that a $\Psi$-configuration does not form a $C_4$.

\begin{proof}[Proof of Lemma \ref{evenc4lemma}]

Suppose $n\geq 26$ is even. We first prove the following claim:

\textbf{Claim:} Any four distinct $A, B, C, D \in \M \cup \L$ form either a $C_4$ or a $\Psi$-configuration.

\begin{proof}[Proof of Claim]

% First suppose that $A, B, C, D$ form a $\Psi$-configuration, say with stem $A$ and anchor $x$. Then because $A \cap B, A \cap C, A \cap D \subseteq \{x\}$, there is no way for two of those intersections to have a choice of distinct representatives. But this is a necessity for $A,B,C,D$ to form a $C_4$, and so they do not form one.

Suppose that $A,B,C,D$ do not form a $\Psi$-configuration. We wish to show that $A,B,C,D$ form a $C_4$

First assume that two of the sets, say $A$ and $C$, are complementary. Since the complement of any set is unique and our sets are in $\M \cup \L$, the intersections $A \cap B$ and $A \cap D$ are nonempty. Moreover, because $A \cap C = \emptyset$ and $A,B,C,D$ do not form a $\Psi$-configuration, $A \cap B$ and $A \cap D$ cannot be the same singleton set. Thus, there exist distinct representatives $x_1 \in A \cap B$, $x_2 \in A \cap D$. Similarly, there exist distinct representatives $x_3 \in B \cap C$, $x_4 \in C \cap D$. Clearly $x_1$ and $x_2$ are distinct from $x_3$ and $x_4$, since the first two are contained in $A$ while the second two are contained in $C = [n]\setminus A$. Thus $A, B, C, D$ form a $C_4$.

Now assume that $A, B, C, D$ are pairwise intersecting. Consider all perfect matchings $\{X_1, X_2\}, \{X_3, X_4\}$ (that is, partitions into sets of size 2) of $\{A,B,C,D\}$, and let $\{A,C\}, \{B,D\}$ be the one that minimizes $|(X_1 \cap X_2) \cup (X_3 \cap X_4)|$. We will show that if $A,B,C,D$ do not form a $C_4$ in that cyclic order, then there is another cyclic order of $A,B,C,D$ that forms a $C_4$. To do this, we use Hall's theorem on distinct representatives as we did in Lemma \ref{trianglelemma}. WLOG, the following are the only cases in which $A,B,C,D$ may fail to form a $C_4$ in that cyclic order:

% \textbf{Case 1:} $|(A \cap B) \cup (B \cap C) \cup (C \cap D) \cup (D \cap A)| \le 3$. The intersections in this union must each be a subset of a 3-set $\{x_1,x_2,x_3\}$. Since $|(A \cap C) \cup (B \cap D)|$ is minimal, we also need that $A \cap C$ and $B \cap D$ are subsets of a 3-set $\{y_1,y_2,y_3\}$. This implies that the sets of the form $X \setminus \{x_1, x_2, x_3, y_1, y_2, y_3\}$, for $X \in \{A,B,C,D\}$, are pairwise disjoint. Counting the number of elements in $[n]$ outside of $\{x_1, x_2, x_3, y_1, y_2, y_3\}$, we get the inequality
% \[ 4\left(\frac{n}{2} - 6\right) \le n - 6 ,\]
% from which it follows that $n \le 18$. Since we assumed that $n \ge 20$, we have excluded this possibility.

\textbf{Case 1:} $|(A \cap B) \cup (B \cap C) \cup (C \cap D)| \le 3$. (Note that this case covers when $|(A \cap B) \cup (B \cap C) \cup (C \cap D)| \le 2$ and when $|(A \cap B) \cup (B \cap C) \cup (C \cap D) \cup (D \cap A)| \le 3$.) The intersections in this union must each be a subset of a 3-set $\{x_1,x_2,x_3\}$. By minimality of $|(A \cap C) \cup (B \cap D)|$, $A \cap C$ and $B \cap D$ are subsets of a 3-set $\{y_1,y_2,y_3\}$. It follows that the sets $X \setminus \{x_1,x_2,x_3,y_1,y_2,y_3\}$ for $X \in \{A,B,C\}$ are pairwise disjoint. Counting the number of elements in $[n]$ outside of $\{x_1,x_2,x_3,y_1,y_2,y_3\}$, we get the inequality
\[ 3\left(\frac{n}{2} - 6\right) \le n - 6 ,\]
from which it follows that $n \le 24$. Since we assumed that $n \ge 26$, this is impossible.

\textbf{Case 2:} $|(A \cap B) \cup (C \cap D)| \le 1$. Since our sets are pairwise intersecting, $A \cap B = C \cap D = \{x\}$ for some $x$. Then $x \in A \cap C$ and $x \in B \cap D$. Since we assumed that $|(A \cap C) \cup (B \cap D)|$ is minimal, it follows that $(A \cap C) \cup (B \cap D) = \{x\}$. But then
\begin{align*}
	(A \cup D) \cap (B \cup C) = (A \cap B) \cup (C \cap D) \cup (A \cap C) \cup (B \cap D) = \{x\},
\end{align*}
from which we get that
\begin{align*}
	n + 1 \ge |(A \cup D) \cup (B \cup C)| + |\{x\}| = |A \cup D| + |B \cup C|  \ge \left(\frac{n}{2} + 1\right) + \left(\frac{n}{2} + 1\right) = n + 2,
\end{align*}
a contradiction.

\textbf{Case 3:} $|(A \cap B) \cup (A \cap D)| \le 1$. Similar to case 3, $A \cap B = A \cap D = \{x_1\}$ for some $x_1$. Since $A,B,C,D$ do not form a $\Psi$-configuration, there must be some $x_2 \in A \cap C$ different from $x_1$. Now, $C \cap D$ cannot be $\{x_1\}$ because otherwise $A \cap B = C \cap D = \{x_1\}$, which we showed is an impossible circumstance in Case 2. Moreover, $C \cap D$ cannot contain $x_2$ because otherwise $x_2 \in A \cap D = \{x_1\}$. Thus, there must be some $x_3 \in C \cap D$ different from $x_1$ and $x_2$. Finally, note that
\begin{align*}
	n &\ge |A \cup B \cup D| \\
	&= |A| + |B| + |D| - |A \cap B| - |A \cap D| - |B \cap D| + |A \cap B \cap D| \\
	&\ge \frac{n}{2} + \frac{n}{2} + \frac{n}{2} - 1 - 1 - |B \cap D| + 1 \\
	&= \frac{3n}{2} - 1 - |B \cap D|,
\end{align*}
from which we get that $|B \cap D| \ge \frac{n}{2} - 1 \ge 4$. So there must be some $x_4 \in B \cap D$ different from each of $x_1,x_2,x_3$. It follows that $A,C,D,B$ form a $C_4$ in that cyclic order.

This concludes the proof of the claim.

\end{proof}

Now we prove the statements of Lemma \ref{evenc4lemma}. Observe that, similar to Lemma \ref{trianglelemma}, statement 2 follows from statement 1, and statement 4 follows from statement 3. So we prove statements 1 and 3. 

Statement 3 follows immediately from the observations about $\Psi$-configurations, specifically that their stem must be an $\frac{n}{2}$-set, and their non-stem sets must be subsets of an $(\frac{n}{2} + 1)$-set, say $X$. There is only one set in $\L$ that could be part of such a configuration, namely $X$ itself. Thus, it is impossible for distinct $M_1, M_2 \in \M, L_1, L_2 \in \L$ to form a $\Psi$-configuration. By the claim, they must form a $C_4$.

For statement 1, first consider the sets $M_1, M_2, M_3, M_4$. If they form a $C_4$, then we are done; otherwise, they must form a $\Psi$-configuration. Say that $M_1$ is the stem. Next consider the sets $M_1, M_2, M_3, M_5$. Again we are done if they form a $C_4$; otherwise, they must form another $\Psi$-configuration. $M_1$ must again be the stem because $M_1 \cap M_2$, $M_1 \cap M_3$ have at most one element, and we have seen that the non-stem sets of $\Psi$-configurations must pairwise intersect in at least $\frac{n}{2} - 1$ elements. So finally consider the sets $M_2, M_3, M_4, M_5$. They are all non-stem sets in our previous two configurations, so $|M_i \cap M_j| \ge \frac{n}{2} - 1$ for all distinct $i,j \in \{2,3,4,5\}$. Thus $M_2,M_3,M_4,M_5$ form a $C_4$.

\end{proof}

Note that in case 1 of the proof of the claim, the required lower bound on $n$ of 26 is not tight because the $x_i$'s and $y_i$'s considered in the proof may not all be distinct. This bound can definitely be reduced, but doing so requires extra casework.

Now we prove Lemma \ref{oddc4lemma} from Lemma \ref{evenc4lemma}.

\begin{proof}[Proof of Lemma \ref{oddc4lemma}]

Let $n \geq 27$ be odd, and let $L_1, L_2, L_3, L_4 \in \L$, meaning that $|L_i| \geq \frac{n+1}{2}$. We break into two cases.

Suppose there exists $j \in [n]$ such that $j$ is in at most two of the $L_i$. Without loss of generality, $j \not\in L_3, L_4$. This means that $L_3$ and $L_4$ are sets of size at least $\frac{n+1}{2} = \frac{n-1}{2} + 1$ contained in a set of size $n-1$ (namely, $[n] \setminus \{j\}$). Let $M_1 \subseteq L_1 \setminus \{j\}$ and $M_2 \subseteq L_2 \setminus \{j\}$ be distinct sets of size $\frac{n-1}{2}$, which are necessarily contained in the same set of size $n-1$ as before (namely, $[n] \setminus \{j\}$). We may then consider $L_3$ and $L_4$ to be in $\L$ and $M_1$ and $M_2$ to be in $\M$ in the sense that Lemma \ref{evenc4lemma}(3) applies in $[n]\setminus \{j\}$: since these sets are distinct, they form a $C_4$.

Otherwise, suppose every $j \in [n]$ is in at least three of the $L_i$. This implies $\sum |L_i| \geq 3n$. No three of the $L_i$ can have size exactly $\frac{n+1}{2}$, since this would imply that the fourth set has size at least $3n - 3\frac{n+1}{2} = \frac{3n}{2} - \frac{3}{2} > n$, an impossibility. Thus at most two of the $L_i$ have size $\frac{n+1}{2}$, meaning that (as in the preceding paragraph) upon the removal of any vertex there are at most two sets of size $\frac{n-1}{2}$. In a similar fashion as the previous paragraph, Lemma \ref{evenc4lemma}(3) implies that these sets form a $C_4$.

\end{proof}

\section{Proof of Theorem \ref{asymcycle}}

Here we construct a partition of $\P(n)$ where almost all of the sets are in parts of size $2(k-1)$. In fact, these parts of size $2(k-1)$ consist of $k-1$ sets of size less than $n/2$, and $k-1$ sets of size at least $n/2$, in such a way that all of the larger sets are disjoint from the smaller sets. The sets not in parts of size $2(k-1)$ can be placed arbitrarily in parts of size at most $k-1$. This partition is $G$-free since the intersection graph of any partition class has connected components with at most $k-1$ vertices. We assume that $n\ge 2(k-1)$.

Define $A_{m,r} := \{A \in \binom{[n]}{m} : \sum_{a \in A} a \equiv r \Mod n\}$. Since $\sum_{r=0}^{n-1} |A_{m,r}| = \binom{n}{m}$, there exists some $r_m$ such that $|A_{m,r_m}| \geq \frac{1}{n}\binom{n}{m}$. Fix these $r_m$ for $k-1 \leq m < n/2$. We construct a part in our partition from each $A \in A_{m,r_m}$ for $k-1 \leq m < n/2$.

Let $A \in A_{m,r_m}$ and enumerate $A=\{a_0,\dots,a_{m-1}\}$  and $B=[n]\setminus A=\{b_0,\dots,b_{n-m-1}\}$. For integers $i$ with $0\leq i\leq\lfloor\frac{m}{k-1}\rfloor-1$, construct the part consisting of the sets
\[ A\setminus\{a_{(k-1)i}\}, A\setminus\{a_{(k-1)i+1}\},\dots, A\setminus\{a_{(k-1)i+(k-2)}\} ,\]
\[ B\setminus\{b_{(k-1)i}\}, B\setminus\{b_{(k-1)i+1}\},\dots, B\setminus\{b_{(k-1)i+(k-2)}\} .\]

The sets of the form $A \setminus \{a_j\}$ (in the first line) are all different. Indeed, suppose $A\setminus\{a_j\}=A'\setminus\{a'_{j'}\}$, so necessarily $|A|=|A'|$.  Also, $\sum_{a_i\in A\setminus\{a_j\}}a_i\equiv\sum_{a'_i\in A'\setminus\{a'_{j'}\}}a_i' \Mod n$.  This implies $-a_j+\sum_{a_i\in A}a_i\equiv-a'_{j'}+\sum_{a'_i\in A'}a'_i \Mod n$.  By construction, this is equivalent to $r_{|A|}-a_j\equiv r_{|A'|}-a'_{j'} \Mod n$, and thus $a_j\equiv a'_{j'} \Mod n$. This means that $a_j=a'_{j'}$, which together with $A\setminus\{a_j\}=A'\setminus\{a'_{j'}\}$ implies that $A=A'$. Thus the two sets were the same. Analogous reasoning concludes that the sets appearing in the second line are also all different. Finally, for any $m$, sets in the first line have size less than ${n\over 2}-1$, while in the second line the sets have size at least ${n\over 2}-1$. Therefore the constructed sets are all different in any part.

For each $A \in A_{m,r_m}$, there are $\floor{\frac{m}{k-1}}$ possible values of $i$ in the construction, that is, $\floor{\frac{m}{k-1}}$ different parts that $A$ generates. Since $\binom{n}{\floor{n/2}} = \Theta(2^n/\sqrt{n})$, this construction creates at least
$$\sum_{m=k-1}^{\lceil\frac{n}{2}\rceil-1}\left\lfloor\frac{m}{k-1}\right\rfloor\frac{1}{n}\binom{n}{m}\geq\frac{2^n}{2(k-1)}\left(1 - \frac{c}{\sqrt{n}}\right)$$
parts of size $2(k-1)$, for some constant $c>0$ not depending on $n$ or $k$.

Since this construction yields at least $\frac{2^n}{2(k-1)}\left(1-c/\sqrt{n}\right)$ parts of size $2(k-1)$, there are at least $2^n\left(1-c/\sqrt{n}\right)$ sets placed in parts this way. Thus at most $2^n-2^n(1-c/\sqrt{n})=2^n\left(c/\sqrt{n}\right)$ sets have not been placed into a part. We place these remaining sets arbitrarily into parts of size $k-1$ (with one possible smaller part).  Partitioning the rest this way generates at most $\frac{2^n}{k-1}\left(c/\sqrt{n}\right)+1$ additional parts.  Thus, in total the partition will have at most $\frac{2^n}{2(k-1)}\left(1-c/\sqrt{n}\right)+\frac{2^n}{k-1}\left(c/\sqrt{n}\right)+1=\frac{2^n}{2(k-1)}\left(1+\Theta\left(1/\sqrt{n}\right)\right)$ parts.

\section{Bounds on $f(6, S_3)$ and $f(9, S_3)$}\label{designs}

\begin{proposition}$f(6,S_3)\le 15$.
\end{proposition}
\begin{proof}
Let $X=\{123,456,12,13,23,45,46,56\}$ be one ($3$-regular but claw-free) partition class. All other classes will be $2$-regular (thus automatically claw-free). Let $Y_1,Y_2,Y_3,Y_4$ contain two pairs of complementary triples, not using the pair $123,456$. Then define $$Z_1=\{14,23456,12356\}, Z_2=\{25,13456,12346\}, Z_3=\{36,12456,12345\}.$$
Let $U_1,U_2,U_3,U_4,U_5$ be defined as the complementary sets of the $1$-factors in a $1$-factorization of $K_6$. Then $W$ is defined by the edges of the $6$-cycle $1,5,3,4,2,6,1$ and $R$ contains $[6]$ together with the one complementary pair of triples not used in $X$ and in $Y_i$. Now we have 15 claw-free partition classes of $\P^*(6)$.
\end{proof}

\begin{proposition} $f(9,S_3)\le 126$.
\end{proposition}

\begin{proof} Take a partition $Q$ of $[9]\choose 3$ into $28$ classes, each containing three pairwise disjoint triples - a very special case of Baranyai's theorem \cite{BA}. However, we need another property of $Q$: four of these classes $X_1,X_2,X_3,X_4$ must form a Steiner triple system. Then these can be extended by the nine pairs covered by their triples implying that  $X_1\cup X_2 \cup X_3\cup X_4$ covers each pair of $[9]$ exactly once. The existence of these $X_i$-s certainly follows from a much stronger result, stating that $[9]\choose 3$ can be partitioned into seven Steiner triple systems (but probably there are easier ways to get them). Then the $X_i$-s provide four claw-free $3$-regular partition classes. Partition classes $Y_i$ can be defined by  putting together $12$  pairs of the remaining $24$ classes of $Q$, they form a double cover of $[9]$. Next we can define $28$ partition classes $Z_i$ by the complements of the $28$ classes of $Q$, each of them forms a double cover on $[9]$.

Next we design $9$ double covers of type $(5,5,8)$ and $18$ double covers of type $(4,7,7)$. To prepare, set $A_i=\{i+1,i+2,i+3,i+6\}$, $B_i=\{i+4,i+5,i+7,i+8\}$ with arithmetic mod 9. Then $9$ double covers of $[9]$ are defined as $U_i=\{A_i\cup i,B_i\cup i, A_i\cup B_i\}$. Set $$C_i=[9]\setminus \{i+1,i+2\}, D_i=[9]\setminus \{i+3,i+6\},$$ $$ E_i=[9]\setminus \{i+4,i+8\}, F_i=[9]\setminus \{i+5,i+7\}.$$ Then $2\times 9$  double covers of $[9]$ are defined as $W_i=\{A_i,C_i,D_i\}$ and $R_i=\{B_i,E_i,F_i\}$.  Note that $U_i,W_i,R_i$ take care of $18$ complementary pairs of sizes $4$ and $5$. The remaining ${9\choose 4}-18$ such pairs can be placed into $54$ partition classes $T_i$ forming double covers on $[9]$.  Finally, $[9]$ alone forms a partition class (leaving some hope of improvement).

Altogether we have $4+12+28+9+18+54+1=126$ claw-free partition classes of $\P^*(9)$.
\end{proof}

\section{Acknowledgement}

This paper was written under the auspices of the Budapest Semesters in Mathematics program during the Fall semester of 2018.

\end{document}